\documentclass[11pt,]{article}
\usepackage{graphicx}
\usepackage{amssymb}
\usepackage{epstopdf}
\usepackage{amsmath,amsthm,amscd,amssymb}
\usepackage{latexsym}
\newtheorem{Theorem}{Theorem}[section]

\newtheorem{Lemma}[Theorem]{Lemma}

\begin{document}

\title{Sum-Product type Estimates over Finite Fields}

\author{Esen Aksoy Yazici}

\maketitle

\begin{abstract}
 Let $\mathbb{F}_q$ denote the finite field with $q$ elements where $q=p^l$ is a prime power. Using Fourier analytic tools with a third moment method, we obtain sum- product type estimates for subsets of $\mathbb{F}_q$.
 
  In particular, we prove that if $A\subset \mathbb{F}_q$, then

$$|AA+A|,|A(A+A)|\gg\min\left\{q,  \frac{|A|^2}{q^{\frac{1}{2}}}  \right\},$$

so that if $A\ge q^{\frac{3}{4}}$, then $|AA+A|,|A(A+A)|\gg q$.
 
\end{abstract}

%\keywords{finite fields, sum-product, third moment}

\section{Introduction} 

Let $R$ be a ring. For a finite subset $A$ of $R$ we define the sum set and the product set of $A$ by
$$A+A=\{a+b:\; a, b \in A\}\; \textit{and}\; AA=\{a.b:\; a,b\in A\},$$
respectively.  It is  expected that, if $A$ is not subring of $R$, then either $|A+A|$ or $|A.A|$ is large compared to $|A|$.

In \cite{ES}, Erd\H{o}s and Szemeredi  proved that that there exists an absolute constant $\epsilon>0$ such that
$$\max\{|A+A|,|AA|\}\gg |A|^{1+\epsilon}$$ holds for any finite subset $A$ of $\mathbb{Z}$. They also conjectured that this bound should hold for any $\epsilon<1$. The best known bound in this direction is due to Shakan \cite[Theorem 1.2]{S} which states that  if $A$ is a finite subset of $\mathbb{R}$, then

$$|A+A|+|AA|\gg  |A|^{\frac{4}{3}+\frac{5}{5277}  }$$

The sum-product problem in the finite field context has been studied by various authors. In this setting, one generally works either on the small sets in terms of the characteristic $p$ of $\mathbb{F}_q$ or for sufficiently large subsets of $\mathbb{F}_q$ to guarantee that the set itself is in fact not a proper subfield of $\mathbb{F}_q$. We refer the reader \cite{AMRS,BKT, HH, MP, P, NRS, R0} and the references therein for an extensive exploration of the problem in this context.

In the present paper, we turn our attention to sum-product type estimates for the sets of the form $BA+C=\{ba+c: a\in A, b\in B, c\in C\}$ and $B(A+C)=\{b(a+c): a\in A, b\in B, c\in C\}$ where $A,B$, and $C$ are subsets of $\mathbb{F}_q$. To estimate a  lower bound for these sets, we first consider an additive energy which we relate with a third moment method. Then we employ a lemma from \cite{BHIPR} and prove the main result in the paper using the tools in  Fourier analysis.

\vskip.125in

\subsection{Preliminaries}
Let $f:\mathbb{F}_q\to \mathbb{C}.$ The Fourier transform of $f$ is defined as 

$$\widehat{f}(m)=q^{-1}\sum_{x\in \mathbb{F}_q}\chi(-xm)f(x)$$
where $\chi(z)=e^{\frac{2\pi iz}{q}}$.
We will use the orthogonality relation 
\begin{equation*} 
  \sum_{x\in \mathbb {F}_{q}}\chi(xs) = \left\{
    \begin{array}{rl}
      q, & \text{if } s=0\\
      0, & otherwise
      \end{array} \right.  
\end{equation*}

and Plancherel identity

$$\sum_{m\in \mathbb{F}_q}|\widehat{f}(m)|^2=q^{-1}\sum_{x\in \mathbb{F}_q}|f(x)|^2. $$

The main result of the paper is the following theorem.

\begin{Theorem}\label{main}
If $A, B,C \subset \mathbb{F}_q$, then 

$$|BA+C|,|B(A+C)|\gg \min \left\{q, \frac{|B|^{\frac{1}{2}} |C|^{\frac{1}{2}} |A|}    {q^{\frac{1}{2}}}  \right\}$$

In particular, taking $A=B=C,$ we have

$$|AA+A|,|A(A+A)|\gg\min\left\{q,  \frac{|A|^2}{q^{\frac{1}{2}}}  \right\}.$$

so that if $A\ge q^{\frac{3}{4}}$, then $|AA+A|,|A(A+A)|\gg q$.

\end{Theorem}

\subsection{Proof of Theorem \ref{main}}

Let $A\subset \mathbb{F}_q$ and $P$ be a set of points in  $\mathbb{F}_q^2\setminus \{(0,0)\}.$ Define the set of lines pinned at $P$ as
$$L=L_P=\{l_{m,b}:(m,b)\in P\}$$
and also the image set of lines in $L$ as
$$L(A)=L_P(A)=\{l_{m,b}(a)=ma+b:(m,b)\in P, a\in A\}.$$

Similar to energy notion given in \cite{AMRS}, define

$$E_3(L,A)=|\{m_1a_1+b_1=m_2a_2+b_2=m_3a_3+b_3 \}|$$
where $(m_1,b_1),(m_2,b_2),(m_3,b_3)\in P, a_i\in A.$\\

\begin{Lemma}\label{relation} With the notation above we have
$$\frac{|L|^3|A|^3}{|L(A)|^2}\le E_3(L,A).$$
\end{Lemma}

\begin{proof}
Let $r(x)=r_{L(A)}(x)=|\{((m,b),a)\in P\times A: x=ma+b \}|$. Then, by H\"older inequality,
\begin{eqnarray*}
|L||A|=\sum r(x)&\le& (\sum r(x)^3)^{\frac{1}{3}}(\sum 1^{\frac{3}{2}})^{\frac{2}{3}}\\
&\le& E_3(L,A)^{\frac{1}{3}}L(A)^{\frac{2}{3}} \\ 
 \frac{|L|^3|A|^3}{|L(A)|^2}&\le& E_3(L,A)
 \end{eqnarray*}

\end{proof}

We need the following lemma from \cite{BHIPR}.

\begin{Lemma}\emph{\cite[Lemma 2.1]{BHIPR}}\label{sum}

$F$ a finite space, $f: F\to R$.

$$\sum_{z\in F}f^{n}(z)\le |F|\left(\frac{||f||_1}{|F|}\right)^n+\frac{n(n-1)}{2}||f||_{\infty} ^{n-2}\sum_{z\in F}\left( f(z)-\frac{||f||_1}{|F|}  \right)^2 $$

where $||f||_1=\sum_{z\in F}|f(z)|$, $||f||_{\infty}=\max_{z\in F}f(z)$.
\end{Lemma}

\begin{Theorem}\label{la} Let $L=L_P$ where $P\cong B\times C$.
$$L(A)\gg \min\left\{q,  \frac{|L|^{\frac{1}{2}}|A|  }{q^{\frac{1}{2}}}  \right\}.$$
\end{Theorem}

\begin{proof}

Now let $$f(z):=r_{L(A)}(z)=|\{z=a_1a_2+a_3: (a_1,a_3)\in P, a_2\in A\}|$$
Then by taking $n=3$, $F=\mathbb{F}_q$  in Lemma \ref{sum}, we have

\begin{eqnarray*}
E_3(L,A)=\sum_{z}f(z)^3\leq\frac{||f||_1^3}{q^2}+3||f||_{\infty}\sum_{z\in F}\left(f(z)-\frac{||f||_1}{q}\right)^2.
\end{eqnarray*}

Note that $||f||_1=\sum f(z)=|L||A| $, $||f||_{\infty}=sup_{z}f(z)\le|L|$, since when we fix $(a_1,a_3)$ in $f(z)$ then $a_2$ is uniquely determined.

Therefore,
%\begin{eqnarray*}
%E_3(L,A)\leq\frac{|L|^3|A|^3}{q^2}+3|L|\sum_{z} (f(z)-\frac{|L||A|}{q})^2
%\end{eqnarray*}

\begin{eqnarray}\label{plancherel}
E_3(L,A)&\leq&\frac{|L|^3|A|^3}{q^2}+3|L|\sum_{z} (f(z)-\frac{|L||A|}{q})^2\nonumber\\
&=&\frac{|L|^3|A|^3}{q^2}+3|L|q\sum_{\xi\ne 0}| \widehat{f(\xi)}|^2
\end{eqnarray}
where we used  the Plancherel in the last equality.

We can write 
\begin{eqnarray*}
f(z)&=&|\{z=a_1a_2+a_3: (a_1,a_3)\in P=B\times C, a_2\in A\}|\\
&=& q^{-1}\sum_{s, a_1,a_2,a_3}\chi ((z-(a_1a_2+a_3))s)B(a_1)C(a_3)A(a_2)\\
&=&q^{-1}\sum_{s, a_1,a_2,a_3}\chi(zs-a_1a_2s)\chi(-a_3s)C(a_3)B(a_1)A(a_2)\\
&=&\sum_{s, a_1,a_2}\chi(zs-a_1a_2s)\widehat{C}(s)B(a_1)A(a_2)
\end{eqnarray*}

It follows that
\begin{eqnarray*}
\widehat{f}(\xi)&=&q^{-1}\sum_{z}\chi(-z.\xi)f(z)\\
&=&q^{-1}\sum_{z}\chi(-z.\xi)\sum_{s, a_1,a_2}\chi(zs-a_1a_2s)\widehat{C}(s)B(a_1)A(a_2)\\
&=&q^{-1}\sum_{s, a_1,a_2}\chi(-a_1a_2s)\widehat{C}(s)B(a_1)A(a_2)\sum_{z\in \mathbb{F}_q}\chi(z(s-\xi))\\
&=&\sum_{a_1,a_2}\chi(-a_1a_2\xi)\widehat{C}(\xi)B(a_1)A(a_2).
\end{eqnarray*}

Therefore,
\begin{eqnarray*}
|\widehat{f}(\xi)|\le \sum_{a_1\in B}|\sum_{a_2\in A}\chi(-a_1a_2\xi)\widehat{C}(\xi)|
\end{eqnarray*}

By the Cauchy-Schwarz inequality, for $\xi\ne 0$,

\begin{eqnarray*}
|\widehat{f}(\xi)|^2&\leq&|B|\sum_{a_1\in \mathbb{F}_q}\sum_{a_2,a_2'\in A}\chi(-a_1a_2\xi)\widehat{C}(\xi)\chi(a_1a_2'\xi)\overline{\widehat{C}(\xi)}\\
&=&|B|\sum_{a_1\in \mathbb{F}_q} \sum_{a_2,a_2'\in A} \chi (\xi a_1(a_2'-a_2) )|\widehat{C}(\xi)|^2\\
&\le&|B|\sum_{a_1\in \mathbb{F}_q} \sum_{a_2,a_2'\in A} \chi (\xi a_1(a_2'-a_2) )|\widehat{C}(\xi)|^2\\
&=&|B|q\sum_{\substack{a_2,a_2'\in A\\ \xi(a_2'-a_2)=0}}|\widehat{C}(\xi)|^2\\
&=&|B|q\sum_{a_2'=a_2\in A}|\widehat{C}(\xi)|^2\\
&=&|B|q|A||\widehat{C}(\xi)|^2
\end{eqnarray*}

It follows that
\begin{eqnarray*}
\sum_{\xi\ne 0}|\widehat{f}(\xi)|^2&\leq&|B|q|A|\sum_{\xi\ne 0}|\widehat{C}(\xi)|^2\\
&\leq&|B|q|A|q^{-1}\sum_{x}|C(x)|^2\\
&=&|B||A||C|\\
&=&|L||A|
\end{eqnarray*}

Plugging the last value in (\ref{plancherel}) and using Lemma \ref{relation} we have
$$\frac{|L|^{3}|A|^3}{|L(A)|^2}\leq E_{3}(L,A)\leq \frac{|L|^3|A|^3}{q^2}+3|L|^2|A|q $$

Therefore

$$L(A)\gg \min\left\{q,  \frac{|L|^{\frac{1}{2}}|A|  }{q^{\frac{1}{2}}}  \right\}.$$

\end{proof}

\begin{proof}[Proof of Theorem \ref{main}]
Note that the set $BA+C=L_P(A)$ where  $P=B\times C$. Hence, taking $|L|=|B||C|$ in Theorem \ref{la}, it follows that
$$|BA+C|\gg \min \left\{q, \frac{|B|^{\frac{1}{2}} |C|^{\frac{1}{2}} |A|}    {q^{\frac{1}{2}}}  \right\}$$
Note that $B(A+C)= L_P(A)$ where  $P\cong B\times C$, so the same argument applies.
\end{proof}

\vskip.125in

\subsection*{Acknowledgements}

The author  would like to thank Simon Macourt, Oliver Roche-Newton, Alex Iosevich, Jonathan Pakianathan and Ilya Shkredov for their valuable comments. The author is supported by T\"UB\.ITAK-B\.IDEB 2218 Postdoctoral Research Fellowship.

%$$\frac{|A|^9}{|AA+A|^2}\leq \frac{|A|^9}{q^2}+3|A|^3q$$

%$$|AA+A|\gg\min\{q,\frac{|A|^3}{q^{\frac{1}{2}}}\}  $$

%In particular when $|A|>q^{\frac{1}{2}}$, then $|AA+A|\gg q$. 


\begin{thebibliography}{9}


\bibitem{AMRS} E. Aksoy Yazici, B. Murphy, M. Rudnev, I. Shkredov, {\em Growth Estimates in Positive Characteristic via Collisions},  Int. Math. Res. Not. IMRN 2017, no. 23, 7148–7189. 

\bibitem{BHIPR} M. Bennett, D. Hart, A. Iosevich, J. Pakianathan, M.  Rudnev, {\em Group actions and geometric combinatorics in $\mathbb{F}_q^d$}. Forum Math. 29 (2017), no. 1, 91–110.
 
\bibitem{BKT} J. Bourgain, N. H. Katz, T. Tao, {\em A sum-product estimate in finite fields, and applications}, Geometric \&Functional Analysis GAFA 14, no 1 (2004), 27-57.

\bibitem{ES} P. Erd\H{o}s, E. Szemer\'edi, {\em On sums and products of integers}. Studies in Pure Mathematics. To the memory of Paul Tur\'an, Basel: Birkhuser Verlag, pp. 213-218.

\bibitem{HH} N. Hegyvari, F. Hennecart, {\em Conditional expanding bounds for two-variable functions over prime fields}, European J. Combin., 34(2013), 1365-1382.

\bibitem{MP}B. Murphy, G. Petridis, {\em A second wave of expanders in finite fields}, Combinatorial and additive number theory. II, 215–238, Springer Proc. Math. Stat., 220, Springer, Cham, 2017.

%\bibitem {MNS} B. Murphy, O. Roche-Newton, I. D. Shkredov, {\em Variations on the sum-product problem}. SIAM Journal on Discrete Mathematics 29(1) (2015), 514-540.

%\bibitem{MNS1} B. Murphy, O. Roche-Newton, I. D. Shkredov, {\em Variations on the sum-product problem II}, https://arxiv.org/pdf/1703.09549.pdf

\bibitem{P} B. Murphy, G. Petridis, {\em Products of difference over arbitrary finite fields}, Discrete Anal. 2018, Paper No. 18, 42 pp. 

\bibitem{NRS} O. Roche-Newton, M. Rudnev, I. D. Shkredov, {\em New sum-product type estimates over finite fields} Adv. Math. 293 (2016), 589–605.
%\bibitem{P} G. Petridis, {\em Products of Differences in Prime Order Finite Fields} https://arxiv.org/pdf/1602.02142.pdf, 2016
\bibitem{R0} M. Rudnev.  {\em An improved sum-product inequality in fields of prime order}, Int. Math. Res. Not. IMRN, (16):3693-3705, 2012.

%\bibitem{RSS}M. Rudnev, I. D. Shkredov, S. Stevens, {\em On the Energy Variant of the Sum-Product Conjecture}, https://arxiv.org/pdf/1607.05053.pdf, 2016
\bibitem{RSS} M. Rudnev, G. Shakan, I. D. Shkredov, {\em Stronger sum–product inequalities for small sets}, https://arxiv.org/pdf/1808.08465.pdf


\bibitem{S}G. Shakan {\em On Higher Energy Decompositions and the Sum-product Phenomenen },https://arxiv.org/pdf/1803.04637.pdf

\end{thebibliography}
 \end{document}